\documentclass[11pt]{article}

\usepackage{amsmath,amssymb,epsf}
\usepackage[french]{babel}
\usepackage[applemac]{inputenc}
\advance \textwidth 3cm
\advance \textheight 3cm
\usepackage{amsfonts}
\usepackage{latexsym}
\newtheorem{theorem}{Theorem}
\newtheorem{lemma}{Lemma}
\newtheorem{corollary}{Corollary}

\newtheorem{proposition}{Proposition}

\newtheorem{remark}{Remark}
\newenvironment{proof}[1]{\par\noindent\underline{Proof #1}:\quad}%
{\unskip\nobreak\hfil\penalty50\hskip2em\null\nobreak\hfil%
$\Box$\parfillskip0pt\par\medskip}
\title{ Maximal Eigenvalue and norm of the  product of Toeplitz matrices. Study of a particular case.}

\author{ Philippe Rambour\thanks{Universit\'{e} de Paris Sud,
      B\^atiment 425; F-91405
Orsay Cedex;
tel : 01 69 15 57 28 ; fax 01 69 15 60 19
      \mbox{e-mail : philippe.rambour@math.u-psud.fr}
     }}
\date{}
\begin{document}
\maketitle
  \renewcommand{\abstractname}{Abstract}
     \begin{abstract}
     \textbf{Maximal eigenvalue and norm of the  product of Toeplitz matrices. Study of a particular case}\\
     In this paper we describe the asymptotic behaviour of the spectral norm of the product of two finite Toeplitz matrices as the matrix dimension goes to infinity. 
     These Toeplitz matrices are generated by positive functions  with Fisher-Hartwig singularities of negative order.
     Since we have positive operators it is known that the spectral norm is also the largest eigenvalue of this product.

           \end{abstract}
           
%


\section{Introduction}
If $f\in L^1 (\mathbb T)$ the Toeplitz matrix with symbol $f$ denoted by $T_N(f)$ is the 
$ (N+1)\times (N+1)$ matrix such that
$$\left(T_N (f)\right)_{i+1,j+1} = \hat f(j-i)
\quad \forall i, j \quad 0\le i,j\le N $$ 
(see, for instance, \cite{Bo.2},\cite{Bo.3}).
We say that a function $h$ is regular if 
 $ h \in L^{\infty} (\mathbb T)$ and $h>0$. Otherwise the function $h$ is said singular.
 If $b$ is a regular function  continuous in $e^{i\theta_{r} }$ we call Fisher-Hartwig symbols the functions 
 $$ f(e^{i\theta}) = b(e^{i\theta}) \prod_{r=1}^R \vert e^{i\theta} -e^{i\theta_{r}} \vert ^{2\alpha_{r}}  
 \varphi _{\beta_{r},\theta_{r}}(e^{i \theta})$$
where 
\begin{itemize}
\item [$\bullet$]
the complex numbers $\alpha_{r}$ and $\beta_{r}$ 
 are subject to the constraints $-\frac {1}{2} < \alpha_{r} < \frac {1}{2}$ and 
 \mbox{$-\frac {1}{2} < \beta_{r} < \frac {1}{2}$,}
 \item [$\bullet$]
 the functions $\varphi _{\beta_{r},\theta_{r}}$ are defined as
 $\varphi _{\beta_{r},\theta_{r}}(e^{i \theta}) = e^{i \beta_{r} (\pi +\theta -\theta_{r})}$.
\end{itemize}
The problem of the extreme eigenvalues 
of a Toeplitz matrix is well known (see \cite{GS} and \cite{Avram1}). 
If $ \lambda_{k,N} \,1 \le k \le N+1$ are  the 
eigenvalues of $T_{N}(f)$ with 
$ \lambda_{1,N} \le \lambda_{2,N}\cdots \le 
\lambda_{N+1,N}$ 
 we have 
$$ \lim_{N \rightarrow= +\infty} \lambda_{1,N} = m_{f} \quad \mathrm{and} \quad
 \lim_{N \rightarrow+\infty} \lambda_{N+1,N} = 
M_{f}$$
 with 
$m_{f} = \mathrm {essinf}\,f$ and $M_{f} = \mathrm {essup} \,f$. 
In \cite{BoGr} and \cite{BoVi}
 B\"{o}ttcher and Grudsky on one hand and B\"{o}ttcher and Virtanen in the other hand give an asymptotic 
 estimation of the maximal eigenvalue in the case of one Toeplitz matrix when the symbol has one or several zeros of negative order. In 
 \cite {RS1111} we have obtained the asymptotic 
of the minimal eigenvalue of one Toeplitz matrix when the symbol has
one zero of order $\alpha$ with $\alpha> \frac{1}{2}$.

But estimatig the eigenvalues of the product of two 
Toeplitz matrices is more delicate. Effectively it is clear that a product of Toeplitz matrices is generally not a Toeplitz matrix. 
In the first part of this paper we consider the product $T_N (f_1) T_N(f_2)$ of two Toeplitz 
matrices  where 
$f_1 (e^{i\theta}) = \vert 1- e^{i \theta} \vert^{-2\alpha_1} c_1(e^{i\theta}),$ and 
$f_2 (e^{i\theta}) = \vert 1- e^{i \theta} \vert^{-2\alpha_2} c_2(e^{i\theta})$
with  $0 <\alpha_1,\alpha_2<\frac{1}{2}$ and 
$c_1,c_2 $  are two regular continuous functions on the torus. 
For these symbols we obtain the norm of the matrix $T_N (f_1) T_N(f_2)$. Owing to an important result of Widom (see Lemma \ref{WIDOM} and also \cite{W2}, \cite{W3}, \cite{Wid}, \cite{Bow2}), which connects the norm of an operator and the norm of a matrix. A proof of this result can be found in \cite{BoVi}. Since $T_N (f_1) T_N(f_2)$ is a positive matrix
the norm is also the maximal eigenvalue of this matrix. Hence our main result (see Theorem \ref{PREMIER}) can be also stated as 
\begin{theorem} \label{SOUSPREMIER}
 Let $f_{1}(e^{i \theta}) = \vert 1- e^{i \theta}  \vert ^{-2\alpha_{1}} c_{1} (e^{i \theta}) $ and 
 $f_{2} (e^{i \theta}) = \vert 1-e^{i \theta}  \vert ^{-2\alpha_{2}} c_{2}(e^{i \theta}) $ with $0<\alpha_{1},\alpha_{2}<\frac{1}{2}$ and 
  $c_{1},c_{2} \in L^\infty (\mathbb T)$ continuous and nonzero in 1. Then 
 if $\Lambda_{\alpha_1,\alpha_2,N} $ is the maximal eigenvalue 
 of $T_{N} (f_{1}) T_{N}( f_{2})$ we have 
 $$ \Lambda_{\alpha_1,\alpha_2,N}  =  N^{2\alpha_{1}+2\alpha_{2}} C_{\alpha_{1}}C_{\alpha_{2}} c_{1}(1) c_{2}(1)\Vert K_{\alpha_{1},\alpha_{2}}\Vert + o(N^{2\alpha_{1}+2\alpha_{2}} ).$$
 with 
 $$\forall \alpha\in ]0,\frac{1}{2}[\, C_{\alpha} = \frac{\Gamma (1-2\alpha)\sin(\pi \alpha)} {\pi}$$
   and $K_{\alpha_1,\alpha_2}$ the integral operator on $L^2[0,1]$ with kernel $(x,y)\rightarrow \int _0^1 \vert x-t\vert^{2\alpha_1-1}
  \vert y-t\vert^{2\alpha_2-1}dt $.
    \end{theorem}

  Then we obtain bounds on $\Vert K_{\alpha_1,\alpha_2}\Vert$ 
  which provides bounds on $\Lambda_{\alpha_1,\alpha_2,N}$ (see the theorem \ref{DEUX}).\\
In a second part we apply this result to obtain  the maximal eigenvalue $\Lambda_{\alpha,\beta,N} $
of the more general symbols
\begin{equation} \label{F1F2}
\tilde f_{1} (e^{i \theta}) = \vert 1- e^{i \theta}  \vert ^{-2\alpha} \prod_{j=1}^p \vert e^{i \theta_{j}} - e^{i \theta}  \vert ^{-2\alpha_{j}} c_{1}(e^{i \theta}) 
   \quad \mathrm{and} \quad
  \tilde f_{2} (e^{i \theta}) = \vert 1- e^{i \theta}\vert ^{-2\beta} \prod_{j=1}^q \vert e^{i \theta_{j}} -e^{i \theta}  \vert ^{-2\alpha_{j}} c_{2} (e^{i \theta}) 
  \end{equation}
  with $0<\alpha,\beta<\frac{1}{2} $, $\displaystyle{\alpha > \max_{1\le j \le p} (\alpha_{j})}$,
  $\displaystyle{\beta > \max_{1\le j \le q} (\beta_{j})}$ and where $c_{1},c_{2}$ are two regular functions 
  satisfying precise hypotheses. We obtain 
 $$\Lambda_{\alpha,\beta,N}\sim C N^{2\alpha+2\beta} \Vert K_{\alpha,\beta}\Vert $$
 (see Theorem \ref{TROISIEME} for the expression of $C$).
 \begin{remark}
 To get Theorem \ref{TROISIEME} we give in Lemma \ref{FOURIER} an asymptotic of the Fourier coefficients of the symbols 
 $\tilde f_{1}$ and $\tilde f_{2}$ of (\ref{F1F2}). We may observe that this lemma provides a statement that
 slightly  differs from Theorem 4.2. in \cite{BoVi}.  
 \end{remark}
 This statement will be
 \begin{theorem}
  Put 
  $\sigma = \prod _{j=1}^R \vert \chi-\chi_{j}\vert^{-2\alpha_{j}}  c$ where 
  $\forall j$, $\chi_{j}\in \mathbb T$  and 
  \begin{itemize}
  \item [i)] $0<\alpha_{1}<\frac{1}{2}$
  \item [ii)]
  $\displaystyle{\alpha_{1} > \max_{2\le j \le R} (\alpha_{j})}$ .
  \end{itemize}
  If $c$ is a regular positive function with 
  $c\in A(\mathbb T,r)$ for $1>r>0$ (see the point 2.2) we have 
  $$ \Lambda_{N} \sim  H \Vert K_{\alpha_{1}}\Vert N^{2\alpha_{1}} $$
  where $\Lambda_{N}$ is the maximal eigenvalue of $T_{N}(\sigma)$, 
  $ H = C_{\alpha_{1}} c(\chi_1) \prod_{j=2}^R \vert 1-\chi_{j}\vert {-2\alpha_{j}}$ and 
  $K_{\alpha_{1}}$ is the integral operator on 
  $L^2(0,1)$ with kernel $(x,y) \rightarrow \vert x-y\vert ^{-2\alpha_{1}-1}$.
 \end{theorem}
An important application 
of the knowledge of the maximal eigenvalue 
of the product of two Toeplitz matrices $T_N (f_1)$ and 
$T_N (f_2)$ is the application of the Gärtner-Ellis Theorem to obtain a large deviation principle and(\cite{DemZei}). Here we consider the  case of 
long memory (see also \cite{Sa.Ka.Ta}).
For the application of the Gärtner-Ellis Theorem in the case where $f_1$ and $f_2$ belong to $L^\infty (\mathbb T)$  \cite{B.G.R.1},
\cite{B.G.L.1},\cite{B.B.B.1} are good references. 
\begin{remark}
For the case where $f,g \in L^\infty (\mathbb T)$ is it not true in general that the maximal eigenvalue 
of $T_{N}(f) T_{N}(g)$ goes to $\mathrm{essup} (f g)$. Likewise it is not always true that 
 the minimal eigenvalue of  $T_{N}(f) T_{N}(g)$ goes to $\mathrm{essinf} (fg)$. If we denote
these maximal and minimal eigenvalues by  $\Lambda_{\max,N}$ and $\Lambda_{\min,N}$ 
Bercu, Bony and Bruneau give in \cite {B.B.B.1} an example of two functions
$f,g \in C^0(\mathbb T), g\ge 0$ such that $\displaystyle{\lim_{N\rightarrow +\infty} \Lambda_{\max,N}}$
exists but is greater than $ \sup_{\theta\in \mathbb T}(fg)(\theta)$ and another example  where 
$\displaystyle{\lim_{N\rightarrow +\infty} \Lambda_{\min,N}}$ is defined but is smaller than 
$ \inf_{\theta\in \mathbb T}(fg)(\theta)$. 
However if $f,g \in L^{\infty} (\mathbb T)$ since $\mathrm{essup} (f) \mathrm{essup} (g) -
T_{N}(f) T_{N}(g)$ is a nonnegative operator it is quite easy to obtain, from the results of 
\cite{B.G.R.1}, that 
$$\mathrm{essup} (f) \mathrm{essup} (g) = \mathrm{essup} (f g) \Rightarrow 
\lim_{N\rightarrow +\infty} \Lambda_{\max,N} =\mathrm{essup} (f g).$$
\end{remark}
 \section{Main result}
 In the rest of this paper we denote by $\chi$ the function $ \theta \rightarrow e^{i \theta}$.
 \subsection{Single Fisher-Hartwig singularities.}
 \begin{theorem} \label{PREMIER}
 Let $f_{1} = \vert 1-\chi \vert ^{-2\alpha_{1}} c_{1}$ and 
 $f_{2} = \vert 1-\chi \vert ^{-2\alpha_{2}} c_{2}$ with $0 <\alpha_{1},\alpha_{2}<\frac{1}{2}$ and 
  $c_{1},c_{2} \in L^\infty (\mathbb T)$ that are continuous and nonzero in 1. We have 
 $$ \Vert T_{N} (f_{1}) T_{N}( f_{2}) \Vert =  N^{2\alpha_{1}+2\alpha_{2}} C_{\alpha_{1}}C_{\alpha_{2}} c_{1}(1) c_{2}(1)\Vert K_{\alpha_{1},\alpha_{2}}\Vert + o(N^{2\alpha_{1}+2\alpha_{2}} ).$$
 with 
  $C_{\alpha_{1}}$, $C_{\alpha_{2}}$ and $K_{\alpha_1,\alpha_2}$ as in Theorem \ref{SOUSPREMIER}.

  \end{theorem}
  Now we give a lemma which is useful to prove Theorem \ref{DEUX}.
  \begin{lemma} \label{DUO}
  There exits a constant 
  $H_{\alpha_{1}\alpha_{2}}$
  such that 
 for all $(x,y) \in [0,1]^2$, $x\not=y$
 $$ \vert y-x \vert ^{2\alpha_{1}+2\alpha_{2}-1} \le \int_{0}^1 \vert x-t\vert ^{2\alpha_{1}-1} \vert y-t \vert ^{2\alpha_{2}-1} dt \le H_{\alpha_{1}\alpha_{2}}
 \vert x-y \vert ^{2\alpha_{1}+2\alpha_{2}-1},$$
 with
 $$ H_{\alpha_{1}\alpha_{2}} = \mathbf B (2\alpha_{1},2\alpha_{2}) +
  \int_{0}^{+\infty} (v^{2\alpha_{1}-1} (1+v)^{2\alpha_{2}-1} +
  v^{2\alpha_{2}-1} (1+v)^{2\alpha_{1}-1}) dv$$
  that is also 
  $$  H_{\alpha_{1}\alpha_{2}} = \mathbf B (2\alpha_{1},2\alpha_{2}) + \mathbf B (2\alpha_{2},3-2\alpha_{1}
  -2\alpha_{2}) + \mathbf B (2\alpha_{1},3-2\alpha_{1}-2\alpha_{2}).$$
  \end{lemma}
 Then we have, as corollary of Theorem \ref{PREMIER} 
  \begin{theorem} \label{DEUX}
 With the hypotheses of Theorem \ref{PREMIER}, if  $\gamma_{\alpha_{1},\alpha_{2}}$ is such that   
$\Vert T_{N}(f_{1}) T_{N}(f_{2})\Vert  \sim N^{2\alpha_{1}+2\alpha_{2}} c_{1}(1) c_{2}(1) 
\gamma_{\alpha_{1}\alpha_{2}}$ 
we have the bounds 
  $$ \psi (\alpha_{1}+\alpha_{2}) \frac{C_{\alpha_{1}} C_{\alpha_{2}} }{C_{\alpha_{1}+\alpha_{2}}}
  \le \gamma_{\alpha_{1},\alpha_{2}}\le H_{\alpha_{1}\alpha_{2}}  \frac{C_{\alpha_{1}}C_{\alpha_{2}} }{C_{\alpha_{1}+\alpha_{2}}}\frac{1}{\alpha_{1}+\alpha_{2}},$$
  with $\psi (\alpha) = \frac{1}{2\alpha} 
\left( \frac{2} {4\alpha+1} + 2 \frac{ \Gamma^2 (2\alpha+1)}{\Gamma(4\alpha+2)}\right) ^{\frac{1}{2}}.$
    \end{theorem}
  
  If we consider now the two symbols $f_{1,\chi_{0}}= \vert \chi_{0}-\chi \vert ^{-2\alpha_{1}} c_{1}$ and 
 $f_{2,\chi_{0}} = \vert \chi_{0}-\chi \vert ^{-2\alpha_{2}} c_{2}$ with $\chi_{0}\in \mathbb T$ it is known
  (see \cite{RS04}) 
 that 
 $$ T_{N}(\vert \chi_{0}-\chi \vert ^{-2\alpha} c) = \Delta_{0}(\chi_{0}) T_{N}\left(  \vert 1-\chi \vert ^{-2\alpha} c_{\chi_{0}}\right)\Delta_{0}^{-1}(\chi_{0})$$
 where $c_{\chi_{0}} (\chi) = c (\chi_{0}\chi)$ and where $\Delta_{0}(\chi_{0})$ is the diagonal matrix defined by 
 $\left( \Delta_{0}(\chi_{0}) \right)_{i,j}=0$ if $i \not=j$ and $ \left (\Delta_{0}(\chi_{0})\right) _{i,i}=\chi_{0}^i$.
 Hence we have the following corollary of Theorems \ref{PREMIER} and \ref{DEUX}
 \begin{corollary}\label{CUN}
 With the previous notations and hypotheses we have 
 $$ \Vert T_{N}(f_{1,\chi_{0}}) T_{N}(f_{2,\chi_{0}}) \Vert 
  \sim  N^{-2\alpha_{1}-2\alpha_{2}} C_{\alpha_{1}}C_{\alpha_{2}} c_{1}(\chi_{0}) c_{2}(\chi_{0})\Vert K_{\alpha_{1},\alpha_{2}}\Vert$$

\end{corollary} 

  \subsection{Several Fisher-Hartwig singularities}
  Let $r>0$, we denote by  $A(\mathbb T,r) $ the set $\{ g\in L^1(\mathbb T) \vert \sum_{u\in \mathbb Z} \vert u\vert ^r \vert \hat g (u) \vert <\infty\}$.  We first state the following lemma 
  \begin{lemma} \label{FOURIER}
  Put 
  $\sigma =\displaystyle{ \prod _{j=1}^R \vert \chi-\chi_{j}\vert^{-2\alpha_{j}}  c}$ where 
  $\forall j$, $\chi_{j}\in \mathbb T$  and 
  $\displaystyle{\alpha_{1} > \max_{2\le j \le R} (\alpha_{j})}$. If $c$ is a regular positive function with 
  $c\in A(\mathbb T,r)$ ($1\ge r>0$ if  $\frac{1}{2} >\alpha_{1}>0$ and $ r\ge 2$ if $0 >\alpha_{1}>
  -\frac{1}{2}$)
 we have 
  $$ \widehat {\sigma} (M) =\displaystyle{ C_{\alpha_{1}} c(\chi_1) \prod_{j=2}^R \vert 1-\chi_{j}\vert ^{-2\alpha_{j}} M^{2\alpha_{1}-1}}
  + o(M^{2\alpha_{1}-1})$$
  uniformly in $M$.
  \end{lemma}
  This lemma and the proof of Theorem \ref{PREMIER}
  allow us to obtain 
  \begin{theorem} \label{TROISIEME}
    Let $\tilde f_{1} = \displaystyle{\vert 1-\chi \vert ^{-2\alpha} \prod_{j=1}^p \vert \chi_{j}-\chi \vert ^{-2\alpha_{j}} c_{1}}$
    and 
  $\tilde f_{2} = \displaystyle{\vert 1-\chi \vert ^{-2\beta} \prod_{j=1}^q \vert 
  \tilde\chi_{j}-\chi \vert ^{-2\alpha_{j}} c_{2} }$
  with $0 <\alpha,\beta< \frac{1}{2}$, $\displaystyle{\alpha > \max_{1\le j \le p} (\alpha_{j})}$,
  $\displaystyle{\beta > \max_{1\le j \le q} (\beta_{j})}$, $\chi_{j}\not=1$, $\tilde \chi_{j}\not=1$ and $c_{1},c_{2}$ two regular functions with 
  $c_{1}\in A(\mathbb T,r_{1}), c_{2}\in A(\mathbb T,r_{2})$ for $1\ge r_{1},r_{2}>0$. Then 
  $$\Vert T_{N}(\tilde f_{1}) T_{N}( \tilde f_{2})\Vert \sim C N^{2\alpha+2\beta} \Vert K_{\alpha,\beta}\Vert $$
  with
  $$ C = c_{1}(1) c_{2}(1) C_{\alpha}C_{\beta} \prod_{j=1}^p \vert 1- \chi_{j}\vert ^{+2\alpha_{j}} 
  \prod_{j=1}^q \vert 1- \tilde \chi_{j}\vert ^{+2\beta_{j}}. $$
  \end{theorem}
  With the same hypotheses on $\alpha$ and $\beta$ we can now consider $T_{N}(\tilde f_{1,\chi_{0}}) T_{N}(\tilde f_{2,\chi_{0}})$ 
  with $\tilde f_{1,\chi_{0}} = \displaystyle{\vert \chi_{0}-\chi \vert ^{-2\alpha} \prod_{j=1}^p \vert \chi_{j}-\chi \vert ^{-2\alpha_{j}} c_{1}}$
    and 
  $\tilde f_{2,\chi_{0}} =\displaystyle{ \vert \chi_{0}-\chi \vert ^{-2\beta} \prod_{j=1}^q \vert\tilde \chi_{j}-\chi \vert ^{-2\alpha_{j}} c_{2} }$,
with 
  $\forall j \in \{1, \cdots,p\}$ $\chi_j\not=\chi_0$ and   
  $\forall h \in \{1, \cdots,q\}$ $\tilde\chi_h\not=\chi_0$ .
    We obtain the corollary 
 
    \begin{corollary}\label{CDEUX}
 With the previous notations and hypotheses we have $$ \Vert T_{N}(\tilde f_{1,\chi_{0}}) 
 T_{N}(\tilde f_{2,\chi_{0}}) \Vert
  \sim  N^{2\alpha+2\beta} C_{\chi_{0}}\Vert K_{\alpha_{1},\alpha_{2}}\Vert$$
  with 
  $$ C_{\chi_{0}} =C_{\alpha}C_{\beta} c_{1}(\chi_{0}) c_{2}(\chi_{0}) \prod_{j=1}^p 
  \vert \chi_{0}-  \chi_{j}\vert ^{-2\alpha_{j}} 
  \prod_{j=1}^q \vert \chi_{0}- \chi_{j}\vert ^{-2\beta_{j}}. $$
    \end{corollary}

 \section{Demonstration of Theorem \ref{PREMIER}}
Let us recall the following Widom's result ( see, for instance, \cite{Bow2}). 
 \begin{lemma} \label{WIDOM}
 Let $A_{N}=(a_{i,j})_{i,j=0}^{N-1}$ be an $N \times N$ matrix with complex entries. We denote by 
 $G_{N}$ the integral operator on $L^2 [0,1]$ with kernel 
 $$ g_{N}(x,y) = a_{[Nx], [Ny]}, \quad (x,y) \in (0,1)^2.$$
 Then the spectral norm of $A_{N}$ and the operator norm of $G_{N}$ are related by 
the equality $\Vert A_{N}\Vert = N \Vert G_{N} \Vert.$
 \end{lemma}
 Denote by $K_{N}$ and $K_{\alpha_{1},\alpha_{2}}$ the integral operators on $L^2(0,1)$ with the kernels, defined for $x\not=y$ by 
 $$ k_{N}(x,y) = N^{-2\alpha_{1}-2\alpha_{2}+1} \sum_{0\le u\le N, u \not=[Nx], u \not=[Ny]}
\Bigl \vert [Nx] -u\Bigr \vert ^{2\alpha_{1}-1}
\Bigl\vert [Ny] -u\Bigr \vert ^{2\alpha_{2}-1} $$
 and 
 $$ k_{\alpha_{1},\alpha_{2}} (x,y) = \int_{0}^1 \vert x-t \vert ^{2\alpha_{1}-1} 
 \vert y-t \vert ^{2\alpha_{2}-1} dt.
 $$
 To prove Theorem \ref{PREMIER} we first 
 assume that the following lemma is true.
 \begin{lemma} \label{UNO}
 The operator $K_{N}$ converges to $K_{\alpha_{1},\alpha_{2}}$ in the operator norm 
 on $L^2(0,1)$.
 \end{lemma}
 Assume Lemma \ref{UNO} is true. Suppose $c_{1}=c_{2}=1$.
 Then put $T_{1,N}, T_{2,N}, D_{1,N},D_{2,N}$ the $(N+1) \times (N+1)$ matrices defined by if $k\not=l$
 $$ (T_{1,N})_{(k+1,l+1)}= C_{\alpha_1}   \vert k-l\vert ^{2\alpha_{1}-1}\quad
  (T_{2,N})_{(k+1,l+1)}=C_{\alpha_2}   \vert k-l\vert ^{2\alpha_{2}-1},$$
  and $ (T_{1,N})_{(k+1,k+1)}=0,  (T_{2,N})_{(k+1,k+1)}=0$.
  On the other hand 
 $$ (D_{1,N})_{(k+1,l+1)}= \left(T_{N}(f_{1})\right)_{(k+1,l+1)}-(T_{1,N})_{(k+1,l+1)},$$
$$ (D_{2,N})_{(k+1,l+1)}= \left(T_{N}(f_{2})\right)_{(k+1,l+1)}-(T_{2,N})_{(k+1,l+1)}.$$
 We can remark that $D_{1,N}$ and $D_{2,N}$ are Toeplitz matrices such that
 $(D_{1,N})_{(k+1,l+1)} =o( \vert k-l\vert ^{2\alpha_1-1})$ and 
 $(D_{2,N})_{(k+1,l+1)} =o( \vert k-l\vert ^{2\alpha_2-1})$
 (see\cite{BrDa}) and this implies (see \cite{BoVi}) 
 $$\Vert D_{1,N} \Vert = o(N^{2\alpha_1}) \quad \mathrm{and}
 \quad \Vert  D_{2,N} \Vert = o(N^{2\alpha_2}).$$
 
 Then we have the upper bound
 $$ \Vert T_{N}(f_{1}) T_{N}(f_{2}) - T_{1,N}T_{2,N}\Vert \le \Vert D_{1,N} T_{2,N} \Vert +
 \Vert  D_{2,N} T_{1,N} \Vert + \Vert D_{1,N} D_{2,N} \Vert $$
 and (see \cite{BoVi})
 $$ \Vert D_{1,N} T_{2,N} \Vert \le \Vert D_{1,N}\Vert \Vert T_{2,N}\Vert = o(N^{2\alpha_{1}})  O (N^{2\alpha_{2}}) = o(N^{2\alpha_{1}+2\alpha_{2}})$$
 $$ \Vert D_{2,N} T_{1,N} \Vert \le \Vert D_{2,N}\Vert \Vert T_{1,N}\Vert = o(N^{2\alpha_{2}}) 
 O (N^{2\alpha_{1}}) = o(N^{2\alpha_{1}+2\alpha_{2}}).$$
 $$ \Vert D_{1,N} D_{2,N} \Vert \le \Vert D_{1,N}\Vert \Vert D_{2,N}\Vert = o(N^{2\alpha_{1}}) 
 o (N^{2\alpha_{2}}) = o(N^{2\alpha_{1}+2\alpha_{2}}).$$
 Hence 
 $$ \Vert T_{N}(f_{1}) T_{N}(f_{2}) \Vert = \Vert T_{1,N}T_{2,N}\Vert + o(N^{2\alpha_{1}+2\alpha_{2}}).$$
 Lemma \ref{WIDOM} implies 
 $$\Bigl \Vert \frac{T_{1,N}T_{2,N}}{N} \Bigr\Vert = \Vert N^{2\alpha_{1}+2\alpha_{2}-1}  K_{N}\Vert$$
 and with Lemma \ref{UNO} we obtain 
 $\displaystyle{\lim_{N\rightarrow + \infty} \Vert K_{N}\Vert =\Vert K_{\alpha_{1}\alpha_{2}}\Vert }$ that ends the proof in the case where the regular function equals 1. Now assume that $c_{1},c_{2}$ are  any continuous positive functions in $ L^{\infty} (\mathbb T)$. Let $\tilde c_{1}$ and $\tilde c_{2}$ defined by 
 $\forall j \in \{1,2\} \quad  \tilde c_{j} (\theta) = c_{j}(\theta)$ if $\theta\not=1$ and $  \tilde c_{j} (1) =0$.
 If $\tilde f_{1} =\vert 1-\chi\vert ^{-2\alpha_{1}}\tilde c_{1}$ and 
 $\tilde f_{2} =\vert 1-\chi\vert ^{-2\alpha_{2}}\tilde c_{2}$ we have (see \cite {BoVi}) 
 $$\Vert T_N \tilde f_{1} \Vert = o(N^{2\alpha_{1}}) \quad \Vert \tilde T_N\tilde f_{2} \Vert = o(N^{2\alpha_{2}}).$$
 Hence $\Vert T_N \tilde f_{1}T_N \tilde f_{2} \Vert =o (N^{-2\alpha_{1}-2\alpha_{2}})$. 
 Since  $f_{1} = \vert 1-\chi\vert ^{-2\alpha_{1}} (\tilde c_{1}+c_{1}(1))$ and 
 $f_{2} = \vert 1-\chi\vert ^{-2\alpha_{2}} (\tilde c_{2}+c_{2}(1))$
  we have $$\Vert T_{N}(f_{1}) T_{N}(f_{2})- T_{N} \left( c_{1}(1) \vert 1-\chi\vert ^{-2\alpha_{1}}\right) 
T_{N}\left( c_{2}(1) \vert 1-\chi\vert ^{-2\alpha_{2}}\right)   \Vert =o(N^{2\alpha_{1}+2\alpha_{2}})$$
and we finally get, via the beginning of the proof 
$$ \Vert T_{N}(f_{1}) T_{N}(f_{2}) \Vert = N^{2\alpha_{1}+2\alpha_{2}} C_{\alpha_{1}}C_{\alpha_{2}} c_{1}(1) c_{2}(1)\Vert K_{\alpha_{1}\alpha_{2}}\Vert + o(N^{2\alpha_{1}+2\alpha_{2}} )$$
 which is the expected formula. We are therefore left with proving Lemma \ref{UNO}.
 \begin{proof}{of the lemma \ref{UNO}}

 Fix $\mu$, $0<\mu<1$ sufficiently close to $1$ such that 
 $\mu>\max (1-2\alpha_{1},1-2\alpha_{2}, \frac{1}{2})$.
 Put 
 $$ k^1_{N}(x,y) = 
\left\{
\begin{array}{cc}
 k_{N}(x,y) & \mathrm{if} \quad \vert x-y \vert >N^{\mu-1} ,  \\
  0& \mathrm{otherwise}  
  \end{array}
\right.
$$
$$
k^2_{N}(x,y) = 
\left\{
\begin{array}{cc}
 k_{N}(x,y) & \mathrm{if} \quad \vert x-y \vert < N^{\mu-1} ,  \\
  0& \mathrm{otherwise}  
  \end{array}
\right.
$$

$$ k_{\alpha_{1},\alpha_{2},N}^1(x,y) =
\left\{
\begin{array}{cc}
 k_{\alpha_{1},\alpha_{2}}(x,y) & \mathrm{if} \quad \vert x-y \vert >N^{\mu-1} ,  \\
  0& \mathrm{otherwise.}  
  \end{array}
\right.
$$
$$
k^2_{\alpha_{1},\alpha_{2},N}(x,y) = 
\left\{
\begin{array}{cc}
k_{\alpha_{1},\alpha_{2}}(x,y) & \mathrm{if} \quad \vert x-y \vert < N^{\mu-1} ,  \\
  0& \mathrm{otherwise}.  
  \end{array}
\right.
$$
If we denote by $K^1_{N}, K^1_{\alpha_{1},\alpha_{2},N}, K^2_{\alpha_{1},\alpha_{2},N}$ the integral operator on $L^2 (0,1)$ with the kernels $h_N^1$, $h_N^1$, $k_{N}^1 k_{\alpha_{1},\alpha_{2},N}^1, 
k_{\alpha_{1},\alpha_{2},N}^2 $ respectively. We have 
$$ \Vert K_{\alpha_{1},\alpha_{2}} - K_{N} \Vert \le   \Vert K^1_{\alpha_{1},\alpha_{2},N} - K^1_{N} \Vert
 + \Vert K^2_{\alpha_{1},\alpha_{2},N}\Vert + \Vert K^2_{N} \Vert.$$
 Hence we have to show that 
 $$ \lim_{N \rightarrow + \infty} \Vert K^1_{\alpha_{1},\alpha_{2},N} - K^1_{N} \Vert =  0, 
\lim _{N \rightarrow + \infty} \Vert K^2_{\alpha_{1},\alpha_{2},N}\Vert= 0, 
\lim _{N \rightarrow + \infty}  \Vert K^2_{N} \Vert =0.$$
First we prove the following lemma.

  \begin{lemma} \label{CINQ}
  When $N$ goes to the infinity 
  $\Vert K^1_{\alpha_{1},\alpha_{2},N} - K^1_{N} \Vert \rightarrow  0$
  \end{lemma}
  \begin{proof}{}
 To prove that 
 $  \Vert K^1_{\alpha_{1},\alpha_{2},N} - K^1_{N} \Vert \rightarrow  0$ it suffices 
 to show that 
 $\vert k^1_{N}(x,y) -k^1_{\alpha_{1},\alpha_{2}}\vert$ converges uniformly to zero for 
 $\vert x -y \vert >N^{\mu-1}$. 
First  we may assume that $x<y$ and we consider the case $[Nx]>N^\mu$ and $ [Ny]<N-N^\mu.$
 Next we study the cases $[Nx]\le N^\mu$ and $[Ny]\ge N - N^{\mu}$. 

 For  $\vert x -y \vert >N^{\mu-1}$ we have to consider the difference
 \begin{align*}
 S_{N}(x,y) &=\frac{1}{N} \sum_{u=0, u \not=[Nx], u \not=[Ny]}^N
 \Bigl\vert \frac{[Nx]}{N} -\frac{u}{N}\Bigr \vert ^{2\alpha_{1}-1}\Bigl\vert \frac{[Ny]}{N} -\frac{u}{N}\Bigr\vert ^{2\alpha_{2}-1} \\
 &- \int_{0}^1 \vert x-t \vert ^{2\alpha_{1}-1} 
 \vert y-t \vert ^{2\alpha_{2}-1} dt.
 \end{align*}
Let $S_{i,N}(x,y)$, $1\le i\le 7$ be the following differences
 \begin{align*}
  S_{1,N} (x,y) = \frac{1}{N}\sum_{u=0}^{[Nx]-N^{\mu_{1}}} 
 \Bigl \vert \frac{[Nx]}{N} -\frac{u}{N}\Bigr\vert ^{2\alpha_{1}-1}&
\Bigl \vert \frac{[Ny]}{N} -\frac{u}{N}\Bigr\vert ^{2\alpha_{2}-1} \\
& - \int_{0}^{\frac{[Nx]}{N}-N^{\mu_{1}-1}} \vert x-t \vert ^{2\alpha_{1}-1} \vert y-t \vert ^{2\alpha_{2}-1} dt,
\end{align*}
 \begin{align*} S_{2,N} (x,y) = \frac{1}{N}
 \sum_{[Nx]-N^{\mu_{1}}+1}^{[Nx]-1}
  \Bigl\vert \frac{[Nx]}{N} -\frac{u}{N}\Bigr\vert ^{2\alpha_{1}-1}&
 \Bigl\vert \frac{[Ny]}{N} -\frac{u}{N}\Bigr\vert ^{2\alpha_{2}-1} \\
 &- \int _{\frac{[Nx]}{N}-N^{\mu_{1}-1}+\frac{1}{N}}^{\frac{[Nx]-1}{N}}
\vert x-t \vert ^{2\alpha_{1}-1} \vert y-t \vert ^{2\alpha_{2}-1}dt,
  \end{align*}
\begin{align*}
S_{3,N} (x,y) =\frac{1}{N} \sum_{[Nx]+1} ^{[Nx]+N^{\mu_{2}}}
 \Bigl\vert \frac{[Nx]}{N} -\frac{u}{N}\Bigr\vert ^{2\alpha_{1}-1}
 &\Bigl\vert \frac{[Ny]}{N} -\frac{u}{N}\Bigr\vert ^{2\alpha_{2}-1} \\
 &- \int_{\frac{[Nx]+1}{N}} ^{\frac{[Nx]}{N}+N^{\mu_{2}-1}} \vert x-t \vert ^{2\alpha_{1}-1} \vert y-t \vert ^{2\alpha_{2}-1} dt,
 \end{align*}
\begin{align*}
 S_{4,N} (x,y) =\frac{1}{N}\sum_{[Nx]+1+N^{\mu_{2}}} ^{[Ny]-N^{\mu_{3}}}
 \Bigl\vert \frac{[Nx]}{N} -\frac{u}{N}\Bigr\vert ^{2\alpha_{1}-1}
&\Bigl \vert \frac{[Ny]}{N} -\frac{u}{N}\Bigr\vert ^{2\alpha_{2}-1} \\
& - \int_{\frac{[Nx]+1}{N}+N^{\mu_{2}-1}} ^{\frac{[Ny]}{N}-N^{\mu_{3}-1}} \vert x-t \vert ^{2\alpha_{1}-1} 
 \vert y-t \vert ^{2\alpha_{2}-1}dt,
 \end{align*}
\begin{align*}
 S_{5,N} (x,y) =\frac{1}{N}\sum_{[Ny]-N^{\mu_{3}}+1} ^{[Ny]-1}
  \Bigl\vert \frac{[Nx]}{N} -\frac{u}{N}\Bigr\vert ^{2\alpha_{1}-1}
 &\Bigl \vert \frac{[Ny]}{N} -\frac{u}{N}\Bigr \vert ^{2\alpha_{2}-1} \\
& - \int_{\frac{[Ny]+1}{N}-N^{\mu_{3}-1}} ^{\frac{[Ny]-1}{N}}
  \vert x-t \vert ^{2\alpha_{1}-1} \vert y-t \vert ^{2\alpha_{2}-1} dt,
  \end{align*}
  \begin{align*}
   S_{6,N} (x,y) =\frac{1}{N}\sum_ {[Ny]+1} ^{[Ny]+N^{\mu_{4}}} 
\Bigl\vert \frac{[Nx]}{N} -\frac{u}{N}\Bigr\vert ^{2\alpha_{1}-1}\Bigr \vert
&\Bigl \vert \frac{[Ny]}{N} -\frac{u}{N}\Bigr\vert ^{2\alpha_{2}-1} \\
 &- \int _ {\frac{[Ny]+1}{N}} ^{\frac{[Ny]}{N}+N^{\mu_{4}-1}} 
 \vert x-t \vert ^{2\alpha_{1}-1} \vert y-t \vert ^{2\alpha_{2}-1} dt,
 \end{align*}
\begin{align*} S_{7,N} (x,y) =\frac{1}{N}\sum_ {[Ny]+N^\mu_{4}+1} ^N 
 \Bigl\vert  \frac{[Nx]}{N} - \frac{u}{N}\Bigr\vert ^{2\alpha_{1}-1}
& \Bigl\vert \frac{[Ny]}{N} -\frac{u}{N}\Bigr \vert ^{2\alpha_{2}-1} \\
 &- \int _ {\frac{[Ny]+1}{N}+N^{\mu_{4}-1}} ^1  \vert x-t \vert ^{2\alpha_{1}-1} \vert y-t \vert ^{2\alpha_{2}-1}dt
\end{align*}
 with $ 0<\mu_{1}<\mu$, $ 0<\mu_{2}<\mu $, $ 0<\mu_{3}<\mu $, $ 0<\mu_{4}<\mu.$
 We can remark that 
$$
 S_{1,N}(x,y) \sim \int_{0}^{\frac{[Nx]}{N}-N^{\mu_{1}-1}} 
\left( \frac{[Nx]}{N} -t \right ) ^{2\alpha_{1}-1}
\left( \frac{[Ny]}{N} -t \right) ^{2\alpha_{2}-1} 
-  \left(x-t \right) ^{2\alpha_{1}-1} \left( y-t \right) ^{2\alpha_{2}-1}dt.$$
We may study the two differences 
 $$ 
  S'_{1,N}(x,y) =  \int_{0}^{\frac{[Nx]}{N}-N^{\mu_{1}-1}} 
 \left( \left( \frac{[Nx]}{N} -t \right)^{2\alpha_{1}-1} -\left( x-t \right) ^{2\alpha_{1}-1} \right)
 \left(\frac{[Ny]}{N} -t \right) ^{2\alpha_{2}-1} dt$$ 
 and 
  $$ S^{\prime\prime}_{1,N}(x,y) =  \int_{0}^{\frac{[Nx]}{N}-N^{\mu_{1}-1}}
\left( x-t \right) ^{2\alpha_{1}-1} \left( \left( \frac{[Ny]}{N} -t \right)^{2\alpha_{2}-1}- \left(y-t \right) ^{2\alpha_{2}-1} \right) dt. $$
Since 
$\vert  \frac{[Nx] -x}{x-t} \vert \le N^{-\mu_{1}}$ we have 
$\left( \frac{[Nx]}{N} -t \right)^{2\alpha_{1}-1} -\left( x-t \right) ^{2\alpha_{1}-1} 
= O(N^{-\mu_{1}})$ and 
$$ \vert  S'_{1,N}(x,y) \vert \le O(N^{-\mu_{1}}) 
 \int_{0}^{\frac{[Nx]}{N}-N^{\mu_{1}-1}} 
 \left(\frac{[Ny]}{N} -t \right) ^{2\alpha_{2}-1} dt = O(N^{-\mu_{1}}) =o(1).$$ 
The same method provides 
$$ S^{\prime\prime}_{1,N}(x,y) \vert = O(N^{-\mu}) =o(1).$$ 
 As previously we have now  
 \begin{align*}
 S_{2,N}(x,y) \sim \int _{\frac{[Nx]}{N}-N^{\mu_{1}-1}+\frac{1}{N}}^{\frac{[Nx]-1}{N}}
 &\left( \left( \frac{[Nx]}{N} -t \right ) ^{2\alpha_{1}-1}
\left( \frac{[Ny]}{N} -t \right) ^{2\alpha_{2}-1} \right.\\
&\left.
-  \left(x-t \right) ^{2\alpha_{1}-1} \left( y-t \right) ^{2\alpha_{2}-1}\right)dt.
\end{align*}
Obviously we have to consider the differences 
$$ 
  S'_{2,N}(x,y) =   \int _{\frac{[Nx]}{N}-N^{\mu_{1}-1}+\frac{1}{N}}^{\frac{[Nx]-1}{N}}
   \left( \left( \frac{[Nx]}{N} -t \right)^{2\alpha_{1}-1} -\left( x-t \right) ^{2\alpha_{1}-1} \right)
 \left(\frac{[Ny]}{N} -t \right) ^{2\alpha_{2}-1} dt$$ 
 and 
  $$ S^{\prime\prime}_{2,N}(x,y) =   \int _{\frac{[Nx]}{N}-N^{\mu_{1}-1}+\frac{1}{N}}^{\frac{[Nx]-1}{N}}
  \left( x-t \right) ^{2\alpha_{1}-1} \left( \left( \frac{[Ny]}{N} -t \right)^{2\alpha_{2}-1}- \left(y-t \right) ^{2\alpha_{2}-1} \right) dt. $$
With the main value theorem we can write
$$ S'_{2,N}(x,y) = -(-2\alpha_1 +1)
\left( \frac{[Nx]} {N} -x \right)
\int _{\frac{[Nx]}{N}-N^{\mu_{1}-1}+\frac{1}{N}}
^{\frac{[Nx]-1}{N}} c_{x,N} ^{2\alpha_{1}-2} (t) 
 \left(\frac{[Ny]}{N} -t \right) ^{2\alpha_{2}-1} dt
$$
with 
$c_{x,N} (t) > N^{-1}$ and 
$$  \int _{\frac{[Nx]}{N}-N^{\mu_{1}-1}+\frac{1}{N}}
^{\frac{[Nx]-1}{N}}  \left(\frac{[Ny]}{N} -t \right) ^{2\alpha_{2}-1} dt =
 O(N^{\mu_1-\mu}).
$$
So 
$  S'_{2,N}(x,y) = O (N^{\mu_1-\mu 
)-2\alpha_1+1}).
$
We can remark that $-2\alpha_1+1-\mu <0 \iff -2\alpha_1+1<\mu$. Hence if $-2\alpha_1+1<\mu$ and $\mu_1$
sufficiently little we have 
$ S'_{2,N}(x,y) =o(1)$.
Likewise we have 
  $ S^{\prime\prime}_{2,N} (x,y) = 
  O\left( N^{-1+ (\mu-1) (2\alpha_2-2)}\right)$. Hence
 $ \mu >\frac{-2\alpha_{2}+1}{-2\alpha_{2}+2}\Rightarrow S^{\prime\prime}_{2,N} (x,y) = o(1)$, and since 
 $-2\alpha_{2}+1>\frac{-2\alpha_{2}+1}{-2\alpha_{2}+2}$ we have $S^{\prime\prime}_{2,N} (x,y) = o(1)$.
 
 We prove exactly as previously 
 
 $$ \mu> -2\alpha_{1}+1 \quad \mathrm{and} \quad
  \mu >-2\alpha_{2}+1 \Rightarrow S_{3,N}=o(1)$$
 $$ \mu_{2}>0 \quad \mathrm{and} \quad \mu_{3}>0 \Rightarrow S_{4,N}=o(1)$$
 Swapping $x$ and $y$ we obtain
\begin {itemize}
\item
$ \mu>-2\alpha_{2}+1$ and $\mu> \frac{-2\alpha_{1}+1}{ -2\alpha_{1}+2}$ then 
$S_{5,N}(x,y) =o(1).$
\item
$ \mu>-2\alpha_{2}+1$ and $\mu> \frac{-2\alpha_{1}+1}{ -2\alpha_{1}+2}$ then 
$S_{6,N}(x,y) =o(1).$
\item
$\mu>0$ and $\mu_{4}>0$ then 
$S_{7,N}(x,y) =o(1).$
\end{itemize}
To complete the proof we have still to bound the integrals
$$ \int_{\frac{[Nx]}{N} -N^{\mu_{1}-1}}^{{\frac{[Nx]}{N} -N^{\mu_{1}-1}}+\frac{1}{N}}
\vert x-t\vert^{2\alpha_{1}-1} \vert y-t\vert^{2\alpha_{2}-1} dt,\quad
 \int_{\frac{[Nx]-1}{N} }^{\frac{[Nx]+1}{N} }
\vert x-t\vert^{2\alpha_{1}-1} \vert y-t\vert^{2\alpha_{2}-1} dt$$
$$ \int_{\frac{[Nx]}{N} +N^{\mu_{2}-1}}^{{\frac{[Nx]}{N} +N^{\mu_{2}-1}}+\frac{1}{N}}
\vert x-t\vert^{2\alpha_{1}-1} \vert y-t\vert^{2\alpha_{2}-1} dt,\quad
 \int_{\frac{[Ny]}{N} -N^{\mu_{3}-1}}^{{\frac{[Ny]}{N} -N^{\mu_{3}-1}}+\frac{1}{N}}
\vert x-t\vert^{2\alpha_{1}-1} \vert y-t\vert^{2\alpha_{2}-1} dt$$
$$ \int_{\frac{[Ny]-1}{N} }^{\frac{[Ny]+1}{N} }
\vert x-t\vert^{2\alpha_{1}-1} \vert y-t\vert^{2\alpha_{2}-1} dt,\quad
\int_{\frac{[Ny]}{N} +N^{\mu_{4}-1}}^{{\frac{[Ny]}{N} +N^{\mu_{4}-1}}+\frac{1}{N}}
\vert x-t\vert^{2\alpha_{1}-1} \vert y-t\vert^{2\alpha_{2}-1} dt$$
which are obviously in $o(1)$with the hypotheses on $\mu$.\\
\textbf{ Assume now $1-N^{\mu-1}> y>N^{\mu-1}>x>0.$}
   For this case we have to consider the decomposition 
   $ \displaystyle{S_{N}(x,y) =\sum_{i=1}^6 S_{i,N} (x,y)} $
 with 
$$ S_{1,N}(x,y)=  \frac{1}{N}\sum_{u=0}^{[Nx]-1} 
 \Bigl \vert \frac{[Nx]}{N} -\frac{u}{N}\Bigr\vert ^{2\alpha_{1}-1}
\Bigl \vert \frac{[Ny]}{N} -\frac{u}{N}\Bigr\vert ^{2\alpha_{2}-1} 
 - \int_{0}^{\frac{[Nx]-1}{N}} \vert x-t \vert ^{2\alpha_{1}-1} \vert y-t \vert ^{2\alpha_{2}-1} dt,$$
 
 and $S_{i,N}$ defined as $S_{i+1,N}$ in the previous case.
We still consider the two differences 
$$ 
  S'_{1,N}(x,y) =  \int_{0}^{\frac{[Nx]-1}{N}}
 \left( \left( \frac{[Nx]}{N} -t \right)^{2\alpha_{1}-1} -\left( x-t \right) ^{2\alpha_{1}-1} \right)
 \left(\frac{[Ny]}{N} -t \right) ^{2\alpha_{2}-1} dt$$ 
 and 
  $$ S^{\prime\prime}_{1,N}(x,y) =  \int_{0}^{\frac{[Nx]-1}{N}}
\left( x-t \right) ^{2\alpha_{1}-1} \left( \left( \frac{[Ny]}{N} -t \right)^{2\alpha_{2}-1}- \left(y-t \right) ^{2\alpha_{2}-1} \right) dt. $$
We have 
$$  S'_{1,N}(x,y) \le N^{(\mu-1)(2\alpha_{2}-1)} O\left(( \frac{[Nx]}{N} )^{2\alpha_{1}}
- x^{2\alpha_{1}}\right) = N^{(\mu-1)(2\alpha_{2}-1)} O(N^{-2\alpha_{1}}).$$
We can remark that 
$ (\mu-1) (2\alpha_{2}-1) -2\alpha_{1}<0 \iff \mu > \frac{2\alpha_{1}}{2\alpha_{2}-1} +1$. Since
$1-2\alpha_{1}> \frac{2\alpha_{1}}{2\alpha_{2}-1} +1$ the hypotheses on $\mu$ give 
$S'_{1,N}(x,y) =o(1)$. Moreover $  \left( \left( \frac{[Ny]}{N} -t \right)^{2\alpha_{2}-1}- \left(y-t \right) ^{2\alpha_{2}-1} \right)=O(N^{-\mu})$ and $ S^{\prime\prime}_{1,N}(x,y) =O(N^{-\mu}) =o(1).$ 
The differences $S_{i,N}$ for $2\le i \le 6$ are as in the first case.\\
The case $N-N^\mu <[Ny] <N$ can be tackled identically.

    \end{proof}
   \textbf{Proof of $\Vert K_{\alpha_1,\alpha_2,N} ^2\Vert \rightarrow 0$} \\
   From the lemma \ref{DUO} we have, for $g \in L^2 (\mathbb T)$ et $y \in [0,1]$ 
   \begin{align*}
   K_{\alpha_1,\alpha_2,N}^2 (g) (x) &=& \int_{0}^1 k_{\alpha_{1},\alpha_{2}}(x,y) g(y) dy 
   = \int_{x-N^{\mu-1}}^{x+N^{\mu+1}} f_{\alpha_{1},\alpha_{2}} (x,y) g(y) dy \\
   &\le&  \int_{x-N^{\mu-1}}^{x+N^{\mu+1}} H_{\alpha_{1},\alpha_{2}} 
   \vert x-y\vert ^{2\alpha_{1}+2\alpha_{2}-1} g(y) dy =K^{\prime 2}_{\alpha_{1},\alpha_{2},N}(g) (y)
   \end{align*}
   where $K_{\alpha_1,\alpha_2,N}^{\prime 2}$ is the integral operator on $L^2 (0,1)$ with kernel 
 $$k_{\alpha_1,\alpha_2,N}^{\prime }(x,y) = 
   H_{\alpha_{1},\alpha_{2}} \vert x-y \vert ^{2\alpha_{1}2\alpha_{2}-1} $$ if 
   $\vert x-y \vert <N^{\mu-1}$ and $k^{\alpha_{1},\alpha_{2},\prime }_{N} (x,y)=0$ otherwise. \\
       If $\Vert g\Vert_{2}=1$we have 
   \begin{align*}
   \int _{0}^1 \vert K^2_{\alpha_1,\alpha_2,N} (g) (x) \vert ^2 dx =& 
   = \int_{0}^1  \Bigl \vert \int _{0}^1 k_{\alpha_1,\alpha_2,N}(x,y) g(y) dy \Bigr \vert ^2 dx\\
   &\le  \int_{0}^1  \Bigl (\int _{0}^1 k_{\alpha_1,\alpha_2,N}(x,y) \vert g(y) \vert dy \Bigr)^2 dx\\
   & \le   \int_{0}^1  \Bigl (\int _{0}^1 k_{\alpha_1,\alpha_2,N}^{\prime}(x,y) \vert g(y) \vert dy \Bigr)^2 dx
   \le \Vert K_{\alpha_1,\alpha_2,N}^{\prime 2}\Vert^2.
\end{align*}
Hence $\Vert K_{\alpha_1,\alpha_2,N}^2 \Vert 
   \le \Vert K_{\alpha_1,\alpha_2,N}^{\prime 2} \Vert = O\left( N ^{(\mu-1)(2\alpha_{1}2\alpha_{2})}\right)=o(1)$
(see \cite{BoVi}).

\noindent
 \textbf{Proof of $\Vert K_{N} ^2\Vert \rightarrow 0$} 
 \\
 As in {\cite{BoVi}} we define the integral operator ${\tilde {K}_{N} ^2}$ on $L^2(0,1)$ with the kernel 
 ${\tilde {k}_{N} ^2}$ defined by $k_{N} ^2$ in the staircase-like 
 bordered strip $\vert [Nx]-[Ny]\vert <N^\mu$ and be zero otherwise. On the squares where 
 ${\tilde {k}_{N} ^2} (x,y) -k_{N} ^2 (x,y) \not=0$ 
 we have $\vert [Nx] - [Ny] \vert \sim N^\mu$ and, 
 as for the proof of the lemma \ref{UNO}, for $(x,y)$ in this squares 
$$
\tilde {k}_{N} ^2 (x,y) -k_{N} ^2 (x,y)
\sim \int_{0}^1 \vert x-t \vert ^{2\alpha_{1}-1} \vert y-t \vert ^{2\alpha_{2}-1}  dt 
$$
and always with the lemma 
\ref{DUO}
$$
\vert \tilde {k}_{N} ^2(x,y) -k_{N} ^2 (x,y) \vert \le 
H_{\alpha_{1},\alpha_{2}} 
\vert x-y\vert ^{2\alpha_{1}+2\alpha_{2}-1} 
= O( N^{(\mu-1)(2\alpha_{1}+2\alpha_{2}-1)}).
$$
As the difference ${ \tilde {h}_{N} ^2} (x,y) -h_{N} ^2 (x,y)$ is 
supported in about $4 (N - N^\mu) =O(N)$ squares of side length $\frac{1}{N}$ we have the 
squared Hilbert-Schmidt norm
$$ \Vert {\tilde {K}_{N} ^2} - K_{N} ^2\Vert =
O\left(N \frac{1}{N^2}N^{(\mu-1) (4\alpha_{1}+4\alpha_{2}-2)}\right).$$
If $2\alpha_1 +2\alpha_2 -1>0$ we have
$(\mu -1) (4\alpha_{1}+4\alpha_{2}-2)-1<0$ and 
\begin{equation}\label{JOLI}
 \Vert {\tilde {K}_{N} ^2} - K_{N} ^2\Vert \rightarrow 0.
 \end{equation}
Otherwise since 
$\mu>\frac{1}{2}> \frac{-4\alpha_1-4\alpha_2+1}{-4\alpha_1-4\alpha_2+2}$ we have also (\ref{JOLI}). 

We are therefore with proving $\Vert {\tilde {K}_{N} ^2}\Vert \rightarrow 0$.
Let $B_N$ be the matrix such \\
$\left( B_N\right) _{k+1,l+1} =
C_{\alpha_1} C_{\alpha_2} 
\displaystyle{ \sum_{u=0 , u\not=k, u\not=l} 
\vert k-u \vert ^{2\alpha_1-1} 
\vert l-u \vert^{2\alpha_2 -1}}$
if $ \vert k-l\vert \le N^\mu$ and $\left( B_N\right) _{k+1,l+1} =0$
otherwise. We have to prove the following technical lemma 
\begin{lemma} 
$\exists \, M_{\alpha_{1},\alpha_{2}}>0$ such for $k\not=l$ 
$$ B_{k+1,l+1} \le M_{\alpha_{1},\alpha_{2}} \vert k-l \vert ^{2\alpha_1 + 2\alpha_2-1}$$
\end{lemma}
\begin{proof}{}

Assume $l>k$ and write 
\begin{align*}
&\sum_{u=0, u\not=k, u\not=l} ^N \vert k-u\vert ^{2\alpha_1-1}
\vert l-u \vert ^{2\alpha_2-1} = 
\sum_{u=0, } ^{k-1} \vert k-u\vert ^{2\alpha_1-1}
\vert l-u \vert ^{2\alpha_2-1}+\\
&+\sum_{k+1} ^{l-1} \vert k-u\vert ^{2\alpha_1-1}
\vert l-u \vert ^{2\alpha_2-1}
+ \sum_{l+1} ^{N} \vert k-u\vert ^{2\alpha_1-1}
\vert l-u \vert ^{2\alpha_2-1}.
\end{align*}
The Euler and Mac-Laurin formula provides 
\begin{align*}
& \sum_{u=0, u\not=k} ^{k-1} \vert k-u\vert ^{2\alpha_1-1}
\vert l-u \vert ^{2\alpha_2-1} =\\ 
&=\int _0^{k-1} (k-u)^{2\alpha_1-1}
(l-u)^{2\alpha-1} du + \frac{1}{2} \left( (l-k+1)^{2\alpha_{2}-1}+ k^{2\alpha_{1}-1} l^{2\alpha_{2}-1}
\right) \left( 1+o(1) \right).
\end{align*}
Since $2\alpha_{1}-1<0$ and $2\alpha_{2}-1<0$ one can find easily $M_{1}>0$ such that 
$$\left( (l-k+1)^{2\alpha_{2}-1}+ k^{2\alpha_{1}-1} l^{2\alpha_{2}-1}
\right) <M_{1} (l-k)^{2\alpha_{1}+2\alpha_{2}-1}.$$
And we have also 
\begin{align*} 
\int _0^{k-1} (k-u)^{2\alpha_1-1} (l-u)^{2\alpha-2} du =& (l-k)^{2\alpha_{1} +2\alpha_{2}-1} \int _{1}
^{\frac{k}{l-k}} u^{2\alpha_{1}-1} (1+u)^{2\alpha_{2}-1} du\\
\le & (l-k)^{2\alpha_{1} +2\alpha_{2}-1} \int _{1}
^{+ \infty} u^{2\alpha_{1}-1} (1+u)^{2\alpha_{2}-1} du.
\end{align*}
Analogously one can show that 
\begin{align*}
&\sum_{u=k+1} ^{l-1} \vert u-k\vert ^{2\alpha_1-1}
\vert l-u \vert ^{2\alpha_2-1} =\\ 
&=\int _{k+1}^{l-1} (u-k)^{2\alpha_1-1}
(l-u)^{2\alpha-1} du + \frac{1}{2} \left( (l-k-1)^{2\alpha_{2}-2}+ (l-k-1)^{2\alpha_{1}-1} 
\right) \left( 1+o(1) \right)
\end{align*}
and 
\begin {eqnarray*}
\int _{k+1}^{l-1} (u-k)^{2\alpha_1-1} (l-u)^{2\alpha-1} du &=& 
(l-k)^{2\alpha_{2}-1} \int _{1}^{l-k-1} v^{2\alpha_{1}-1} \left( 1-\frac{v}{l-k}\right) ^{2\alpha_{2}-1} dv \\
&=& (l-k)^{2\alpha_{1}+2\alpha_{2}-1} \int _{\frac{1}{l-k}}^{1-\frac{1}{l-k} } w^{2\alpha_{1}-1} (1-w)^{2\alpha_{2}-1} dw\\
&\le &  (l-k)^{2\alpha_{1}+2\alpha_{2}-1} \int _{0}^{1} w^{2\alpha_{1}-1} (1-w)^{2\alpha_{2}-1} dw.
\end{eqnarray*}
The last sum provides 
\begin{align*}
 \sum_{u=l+1} ^{N} \vert u-k\vert ^{2\alpha_1-1}
\vert l-u \vert ^{2\alpha_2-1} &=\int _{l+1}^N (u-k)^{2\alpha_1-1}
(u-l)^{2\alpha-1} du \\
&+ \frac{1}{2} \left( (l-k+1)^{2\alpha_{1}-2}+ (N-k)^{2\alpha_{1}-1} 
(N-k)^{2\alpha_{2}-1}  \right) \left( 1+o(1) \right).
\end{align*}
We have $$ (N-k)^{2\alpha_{1}-1} (N-k)^{2\alpha_{2}-1} \le (l-k)^{2\alpha_{1}+2\alpha_{2}-2} 
\le  (l-k)^{2\alpha_{1}+2\alpha_{2}-1}
$$
and 
\begin{eqnarray*}
\int _{l+1}^N (u-k)^{2\alpha_1-1} (u-l)^{2\alpha-1} du &=& 
(l-k)^{2\alpha_{2}-1} \int _{l+1-k}^{N-k} v^{2\alpha_{1}-1} \left( \frac{v}{l-k} -1\right)^{2\alpha_{2}-1} dv\\
&=& (l-k)^{2\alpha_{1}+2\alpha_{2}-1} \int _{1+ \frac{1}{l-k}}^{\frac{N-k}{l-k}} w^{2\alpha_{1}-1} 
(w-1)^{2\alpha_{2}-1} dw\\
&\le & (l-k)^{2\alpha_{1}+2\alpha_{2}-1} \int _{1}^{+\infty} w^{2\alpha_{1}-1} 
(w-1)^{2\alpha_{2}-1} dw
\end{eqnarray*}
that ends the proof of the lemma. 
         \end{proof}
         Using lemma \ref {WIDOM} we can write 
         \begin{equation}\label{WIDOM2}
         \Vert \tilde {H}_{N} ^2 \Vert = \frac{1}{N} N^{2\alpha_{1}+2\alpha_{2}+1} 
         \Vert B_{N}\Vert.
         \end{equation}
         Consider now the matrix $C_{N}$ 
         defined by    
           $ \left(C_N\right)_{k+1,l+1}=0$ for $\vert k-l\vert \ge N^\mu$, $ \left(C_N\right)_{k+1,l+1}=
          M_{\alpha_{1}\alpha_{2}}\vert k-l\vert^{-2\alpha_1-2\alpha_2 -1}$ for
$ 0<\vert k-l\vert <N^\mu $, 
$ \left(C_N\right)_{k+1,k+1}= C_{\alpha_1} C_{\alpha_2} \sum_{u=0}^\infty 
 u^{-2\alpha_1-2\alpha_2 -2}$
 if $-2\alpha_{1}-2\alpha_{2}-1<0$,
 $ \left(C_N\right)_{k+1,k+1}=
 2 N^{-2\alpha_{1}-2\alpha_{2}-1} \int_{0}^1 \vert \frac{k}{N} -t \vert ^{-2\alpha_{1}-2\alpha_{2}-2} dt$ if 
 $-2\alpha_{1}-2\alpha_{2}-1>0.$
       If $x (x_{1},\cdots,x_{N+1})$ and $y (y_{1},\cdots, y_{N+1})$ are two vectors of $\mathbb R^{N+1}$ we have 
\begin{align*}
\Bigl \vert \Bigl \langle  B_{N} (x) \vert y \Bigr \rangle \Bigr \vert &= 
\Bigl \vert \sum_{i=1}^{N+1} \left( \sum_{j=1}^{N+1} \left(B_{N}\right) _{i,j} x_{j}\right) y_{i} \Bigr \vert \\
& \le \sum_{i=1}^{N+1} \left( \sum_{j=1}^{N+1} \left(B_{N}\right) _{i,j} \vert x_{j} \vert \right) \vert y_{i} \vert \\
& \le \sum_{i=1}^{N+1} \left( \sum_{j=1}^{N+1} \left(C_{N}\right) _{i,j} \vert x_{j} \vert \right) \vert y_{i} \vert 
\end{align*}
and  $$\Vert B_{N}\Vert \le\Vert  C_{N} \Vert. $$
But
$$\Vert  C_{N} \Vert\le O(\sum_{i=1}^{N^\mu} i^{-2\alpha_{1}-2\alpha_{2}-1}=O\left( N^{\mu (-2\alpha_{1}-2\alpha_{2})} \right).$$
  and from the equality (\ref{WIDOM2}) 
    $$  \Vert \tilde {K}_{\alpha_{1},\alpha_{2},N} ^2 \Vert = O\left( N ^{(2\alpha_{1}+2\alpha_{2}) (1-\mu)}\right)
  $$
 hence $\Vert \tilde {K}_{\alpha_{1},\alpha_{2},N} ^2 \Vert\rightarrow 0$ that achieves the proof of the lemma
  \ref{UNO}.
           \end{proof}
                    \section{ Demonstration of Lemma \ref{DUO}
           and Theorem\ref{DEUX}}
           \subsection{Proof of Lemma \ref{DUO}}
           Assume $y>x$. We have 
           \begin{eqnarray*}
           \int_{0}^x (x-t)^{2\alpha_{1}-1} (y-t)^{2\alpha_{2}-1} dt &=& \int_{0}^x u^{2\alpha_{1}-1} 
           (y-x+u)^{2\alpha_{2}-1} du \\
           &=& (y-x)^{2\alpha_{2}-1} \int_{0}^x  u^{2\alpha_{1}-1} (1+\frac{u}{y-x})^{2\alpha_{2}-1} du \\
           &=& (y-x)^{2\alpha_{1}+2\alpha_{2}-1} 
           \int_{0}^{\frac{x}{y-x}}  v^{2\alpha_{1}-1} (1+v)^{2\alpha_{2}-1} dv.
           \end{eqnarray*}
Consequently          
           $$   \int_{0}^x (x-t)^{2\alpha_{1}-1} (y-t)^{2\alpha_{2}-1} dt \le 
           (y-x)^{2\alpha_{1}+2\alpha_{2}-1} 
           \int_{0}^{\infty}  v^{2\alpha_{1}-1} (1+v)^{2\alpha_{2}-1} dv.$$
           We can also write 
            \begin{eqnarray*}
           \int_{x}^y (t-x)^{2\alpha_{1}-1} (y-t)^{2\alpha_{2}-1} dt &=& \int_{0}^{y-x} u^{2\alpha_{1}-1} 
           (y-x-u)^{2\alpha_{2}-1} du \\
           &=& (y-x)^{2\alpha_{2}-1} \int_{0}^{y-x}  u^{2\alpha_{1}-1} (1-\frac{u}{y-x})^{2\alpha_{2}-1} du \\
           &=& (y-x)^{2\alpha_{1}+2\alpha_{2}-1} 
           \int_{0}^{y-x}  v^{2\alpha_{1}-1} (1-v)^{2\alpha_{2}-1} dv
           \end{eqnarray*}
         and 
         $$ \int_{x}^y (t-x)^{2\alpha_{1}-1} (y-t)^{2\alpha_{2}-1} \le (y-x)^{2\alpha_{1}+2\alpha_{2}-1} 
           \int_{0}^{1 } v^{2\alpha_{1}-1} (1-v)^{2\alpha_{2}-1} dv.
         $$
           Finally we have 
           \begin{eqnarray*} 
           \int_{y}^1 (t-x)^{2\alpha_{1}-1} (t-y)^{2\alpha_{2}-1} dt &=& \int_{0}^{1-y} (u+y-x)^{2\alpha_{1}-1}
           u^{2\alpha_{2}-1} du \\
           &=&(y-x)^{2\alpha_{1}-1}  \int_{0}^{1-y} \left( \frac{u}{y-x} +1\right)^{2\alpha_{1}-1} 
           u ^{2\alpha_{2}-1}du \\
           &=& (y-x)^{2\alpha_{1}+2\alpha_{2}-1} \int_{0}^{\frac{1-y}{y-x}} (v+1)^{2\alpha_{1}-1} 
           v^{2\alpha_{2}-1} dv
           \end{eqnarray*}
           and 
           $$  \int_{y}^1 (t-x)^{2\alpha_{1}-1} (t-y)^{2\alpha_{2}-1} dt \le (y-x)^{2\alpha_{1}+2\alpha_{2}-1}           
           \int_{0}^{+ \infty} (v+1)^{2\alpha_{1}-1} 
           v^{2\alpha_{2}-1} dv.
           $$
           thus it implies that
           $$ \int_{0}^1\vert x-u\vert ^{2\alpha_{1}-1} \vert y-u \vert ^{2\alpha_{2}-1} du \le  
           H_{\alpha_{1}\alpha_{2}} \vert y-x \vert ^{2\alpha_{1}+2\alpha_{2}-1} $$
           with 
           $$ H_{\alpha_{1}\alpha_{2}} = \mathbf B(-2\alpha_{1},-2\alpha_{2}) +
           \mathbf B( 2\alpha_{1}, 3-2\alpha_{1}-2\alpha_{2})
           + \mathbf B( 2\alpha_{2}, 3-2\alpha_{1}-2\alpha_{2}) .$$
           To obtain the lower bound we write, 
          $$
           \int_{0}^x (x-u)^{2\alpha_{1}-1} (y-u)^{2\alpha_{2}-1} du \ge
           \int_{0}^x (y-u)^{2\alpha_{1}+2\alpha_{2}-2} du
           $$
           that is also
           $$   \int_{0}^x (x-u)^{2\alpha_{1}-1} (y-u)^{2\alpha_{2}-1} du \ge 
           \frac{ y^{2\alpha_{1}+2\alpha_{2}-1}- (y-x) ^{2\alpha_{1}+2\alpha_{2}-1}}
           {2\alpha_{1}+2\alpha_{2}-1}$$
           and  $$ \int_{0}^x (x-u)^{2\alpha_{1}-1} (y-u)^{2\alpha_{2}-1} du \ge 0.$$
           Likewise we have 
            $$ \int_{y}^1 (u-x)^{2\alpha_{1}-1} (y-u)^{2\alpha_{2}-1} du \ge 0.$$
                   Since we have also 
            $$ \int_{x}^y (u-x)^{2\alpha_{1}-1} (y-u)^{2\alpha_{2}-1} du \ge (y-x)^{2\alpha_{1}-2\alpha_{2}-1}$$
            we can conclude that 
         $$ \int_{0}^1\vert x-u\vert ^{2\alpha_{1}-1} \vert y-u \vert ^{2\alpha_{2}-1} du \ge  
          \vert y-x \vert ^{2\alpha_{1}+2\alpha_{2}-1}. $$
            \subsection{Proof of Theorem \ref{DEUX}}
           Taking into account that 
                      $$ \int_{0}^1 \vert x-t\vert ^{2\alpha_{1}-1} \vert y-t \vert ^{2\alpha_{2}-1} dt \le H_{\alpha_{1}\alpha_{2}} \vert x-y \vert ^{2\alpha_{1}+2\alpha_{2}-1}$$ we get  $\Vert K_{\alpha_{1},\alpha_{2}}\Vert \le \Vert K_{\alpha_{1}+\alpha_{2}}\Vert$
           where $K_{\alpha_{1}+\alpha_{2}}$ is the integral operator on $L^2(0,1)$ with kernel 
           $(x,y) \rightarrow \vert x-y\vert ^{2\alpha_{1}+2\alpha_{2}-1}$ (see the demonstration of 
           $\Vert K_{\alpha_{1},\alpha_{2},N}\Vert$ goes to zero in the proof of Lemma \ref{CINQ}).
       Using the following proposition (see \cite{BoVi})
       \begin{proposition} \label{encadrement}
  If $f = \vert \chi-\chi_{0}\vert ^{-2\alpha} c$ with $ c\in L^\infty (\mathbb T)$ continuous and nonzero 
  at $\chi_{0}\in \mathbb T$ and $\alpha\in ]0,\frac{1}{2}[$, if $K_\alpha$ 
  is the integral operator on $L^2(0,1)$  with kernel $(x,y)\rightarrow 
  \vert x-y\vert ^{2\alpha-1}$ then we have
  $$\Vert T_{N}(f)\Vert \sim N^{2\alpha} C_{\alpha} \Vert K_{\alpha}\Vert c(\chi_{0})$$
and 
$$ \psi(\alpha)
\le \Vert K_{\alpha}\Vert \le \frac{1}{\alpha}$$
    \end{proposition}
we obtain the upper bound  for $\Vert K_{\alpha_{1}+\alpha_{2}}\Vert$.\\
                      Let $\mathbf{1}$ be the function which is identically $1$ on $[0,1]$. We have, from Lemma \ref{DUO}
           $$ \Vert K_{\alpha_{1},\alpha_{2}} \Vert \ge 
           \frac{\Vert K_{\alpha_{1},\alpha_{2}} \mathbf{1} \Vert }{\Vert \mathbf{1}\Vert}
           = \Vert K_{\alpha_{1},\alpha_{2}} \mathbf{1} \Vert  \ge
            \Vert K_{\alpha_{1}+\alpha_{2}} \mathbf{1} \Vert.$$
            Since
            $ K_{\alpha_{1}+\alpha_{2}} \mathbf{1} (1)
            (x) = \frac{1}{2(\alpha_1+\alpha_2)} 
            \left( x^{2(\alpha_1+\alpha_2)} + (1-x)^{2
            (\alpha_1+\alpha_2)}\right)$, 
            we obtain that 
            $ \Vert K_{\alpha_{1},\alpha_{2}} \Vert$ is greater than or equal to
            $$ \frac{1}{4(\alpha_1+\alpha_2)}
            \int _0^1\left( x^{2(\alpha_1+\alpha_2)} 
            + (1-x)^{2 (\alpha_1+\alpha_2)}\right)^2 dx=
            \psi (\alpha_1,\alpha_2).$$
            This prove the lower bound for 
                   $\Vert K_{\alpha_{1},\alpha_{2}}\Vert$. 
      \section{ Demonstration of Theorem 
      \ref{TROISIEME}}
           \subsection{Demonstration of Lemma \ref{FOURIER}}
                    \subsubsection{First step: one singularity}
           Put $\sigma_{\alpha} = \vert 1-\chi\vert^{-2\alpha}  $ and $\sigma= \vert 1-\chi\vert^{-2\alpha} c$ with $c \in A(r, \mathbb T)$, where $r$ will be precise later.
           First we prove
           $\hat \sigma(M) = \vert M^{2\alpha-1} \vert c(1) \left(1+o(1)\right)$ uniformly in $M$.
            Of course we have for all $M \in \mathbb Z$ 
       \\ $ \displaystyle{\hat \sigma (M) = \sum_{u+v=M} \widehat {\sigma_{\alpha} } (u) \hat c (v)}.$
           Let $\epsilon>0$ and an integer $S_{0}>0$ such that 
           $$ \forall S \vert S \vert  \ge S_{0} \quad \sum_{\vert s \vert \le S_{0}} \hat c (s) = c(1) +R_{S}
         \quad   \mathrm{and} \quad \widehat {\sigma_{\alpha}} (S) = C_{\alpha} \vert S \vert ^{-2\alpha-1} 
         (1+r_{S})
           $$
           with $\vert R_{S}\vert \le \epsilon$ and $\vert r_{S}\vert \le \epsilon$. 
           We have 
           \begin{eqnarray*} 
           \hat \sigma (M) &=& \sum_{v< - S_{0}} \widehat{\sigma_{\alpha}} (M-v) \hat c (v) 
           +  \sum_{S_{0}\ge v\ge  - S_{0}} \widehat{\sigma_{\alpha} }(M-v) \hat c (v) \\
           &+&  \sum_{v> S_{0}} \widehat{\sigma_{\alpha}} (M-v) \hat c (v) 
           \end{eqnarray*}
           Obviously 
           $$ \Bigl \vert\sum_{v< - S_{0}} \sigma_{\alpha} (M-v) \hat c (v)\Bigr \vert  \le \max_{w\in\mathbb Z} 
           \vert\widehat   \sigma_{\alpha} (w)\vert \sum_{v< - S_{0}} \vert \hat c (v) \vert $$
           and if $c \in A(r,\mathbb T)$ and $S_{0}=N^\nu$ $0<\nu<1$ we can conclude 
           $$ \Bigl \vert\sum_{v< - S_{0}} \sigma_{\alpha} (M-v) \hat c (v)\Bigr \vert   =O(N^{-r \nu}).$$
           Now if $\nu$ is such that 
           $-r \nu <2\alpha-1$ we obtain 
            $$ \Bigl \vert\sum_{v< - S_{0}} \sigma_{\alpha} (M-v) \hat c (v)\Bigr \vert =o(N^{2\alpha-1}).$$
            To have $r\nu<2\alpha+1$ with $\nu \in ]0,1[$ and $\alpha\in ]0, \frac{1}{2}[$ we must choose
            $r$ in $]0,1]$. Moreover if  $\alpha\in ]-\frac{1}{2},0[$ we must pick $\alpha$ in $[2, + \infty[$.
           Clearly we have also 
           $$\Bigl \vert  \sum_{v> S_{0}} \widehat{\sigma_{\alpha}} (M-v) \hat c (v) \Bigr \vert
            = o(N^{2\alpha-1}).$$
           Moreover we have, if $\vert M \vert \ge 2S_{0}$, 
           $$  \sum_{S_{0}\ge v\ge  - S_{0}} \widehat{\sigma_{\alpha} }(M-v) \hat c (v) =
       C_\alpha    \vert M \vert ^{2\alpha-1} c(1)
       \left(1+ o(1)\right)$$ that is the announced result.
           \subsubsection{Second step: two singularities}
           With the same notations than previously we can consider the Fourier coefficients of the function
           $\sigma=\sigma_{\alpha_{2}}(\chi_{0}\chi)\sigma_{\alpha_{1}} c$ with $\alpha_{1}<\alpha_{2}$ and $\chi_0 \not=1$. 
           Following the first step we can assume $c=1$ without lost of generality. 
            For all $M \in \mathbb Z$ we have 
            $$\widehat{\sigma_{\alpha_{2}}(\chi_{0}\chi)}= \widehat{\vert 1-\chi_{0}\chi \vert ^{\alpha_{2}} }(M) = 
            \chi_{0}^{-M} \widehat{\sigma_{\alpha_2}}(M).$$
            Let $\epsilon >0$ and $S_0>0$ such that $S>S_0$ implies
           \begin{itemize}
           \item [$\bullet$]
           $$ \sum_{-S\le v\le S}\widehat 
           {\sigma_{\alpha_1}} (v) (\chi_0^{-v}) 
           =\sigma_{\alpha_1} (\chi_0^{-1}) (1+R_1)$$
           with $\vert R_1 \vert <\epsilon$.
           \item [$\bullet$]
                  $$ \sum_{-S\le v\le S} \widehat 
           {\sigma_{\alpha_2} }(v) (\chi_0^{-v}) 
           =\sigma_{\alpha_2} (\chi_0^{-1}) (1+R_2)$$
           with $\vert R_2 \vert <\epsilon$.
           \item [$\bullet$]
          For all $S$ such that $\vert S \vert >S_0$ we have 
          $$ \widehat {\sigma_1} (S)= C_{\alpha_1} 
          \vert S\vert ^{-2\alpha_1-1 } (1 +R_{1,S})$$
          with $ R_{1,S} =O(\epsilon)$.
           \item [$\bullet$]
            For all $S$ such that $\vert S \vert >S_0$ we have 
          $$ \widehat {\sigma_2} (S)= C_{\alpha_2} 
          \vert S\vert ^{-2\alpha_2-1 } (1 +R_{2,S})$$
          with $ R_{2,S} =O(\epsilon)$.
          \end{itemize}
         Since $$\hat \sigma (M) = 
         \sum_{v\in \mathbb Z} \widehat{ \sigma_{\alpha_{1}}}(M-v) \chi_{0}^{-v} \widehat{\sigma_{\alpha_{2}}}(v)$$
         and $$\hat \sigma (-M) = \sum_{v\in \mathbb Z} \widehat{ \sigma_{\alpha_{1}}}(M-v) \chi_{0}^{v} \widehat{\sigma_{\alpha_{2}}}(v)$$
we can assume, without loss of generality, that $M>0$.
  The aim of the rest of this demonstration is  to prove that for $M$ sufficiently large we have the formula
            $$\sigma (M) = C_{\alpha_1} \vert M \vert ^{2\alpha_1-1} c(1)
          \prod_{j=2} ^n \vert \chi_0 - \chi\vert ^{-2\alpha_j} (1+R_M)
          $$ with $\vert R_M \vert =O( \epsilon).$\\ 
           Let $\nu$ be a fixed real such $0<\nu<1$.
            We  write 
            $$
            \hat \sigma(M) =\sum_{i=0}^5 \Sigma_{i}(M).
      $$
             where           $$ \Sigma_{1}(M)= 
            \sum_{v\ge M+M^\nu} \widehat{ \sigma_{\alpha_{1}}}(M-v) \chi_{0}^{-v} \widehat{\sigma_{\alpha_{2}}}(v)
            \quad 
            \Sigma_{2}(M)= 
            \sum_{M-M^\nu<v< M+M^\nu} \widehat{ \sigma_{\alpha_{1}}}
            (M-v)  \chi_{0}^{-v}        
            \widehat{\sigma_{\alpha_{2}}}(v)$$
            $$ \Sigma_{3}(M)= 
            \sum_{M^\nu\le v\le M-M^\nu} \widehat{ \sigma_{\alpha_{1}}}    
            (M-v)\chi_{0}^{-v} 
            \widehat{\sigma_{\alpha_{2}}}(v)
            \quad
            \Sigma_{4}(M)= 
            \sum_{-M^\nu< v\le M^\nu} \widehat{ \sigma_{\alpha_{1}}} 
            (M-v)\chi_{0}^{-v} 
            \widehat{\sigma_{\alpha_{2}}}(v)
            $$
            $$\Sigma_{5}(M)= 
            \sum_{ v\le -M^\nu} \widehat{ \sigma_{\alpha_{1}}}(M-v)
             \chi_{0}^{-v} 
            \widehat{\sigma_{\alpha_{2}}}(v).
            $$
         
                Assume now $\vert M ^\nu\vert >S_0$.
                    We have
       $$
            \Sigma_{1} (M)= C_{\alpha_{1}} C_{\alpha_{2}} \sum_{v\ge M+M^\nu} (v-M)^{2\alpha_{1}-1}
              v^{2\alpha_{2}-1} \chi_{0}^{-v} \left(1+R_1(M)\right)$$
              with $R_1(M)=O(\epsilon)$.
              An Abel summation provides 
           \begin{eqnarray*}
           && \sum_{v\ge M+M^\nu} (v-M)^{2\alpha_{1}-1}
              v^{2\alpha_{2}-1} \chi_{0}^{-v} =
                          \sum_{v\ge M+M^\nu}
            \left ( (v-M)^{2\alpha_{1}-1} v ^{2\alpha_{2}-1} \right. \\
            &-&\left. (v+1-M)^{2\alpha_{1}-1} 
           (v+1) ^{2\alpha_{2}-1}\right) \tau_{v}+
            \left(M^\nu\right)^{2\alpha_{1}-1} 
           (M+M^\nu)^{2\alpha_{2}-1}
           \tau_{S_{0}(M)-1}
           \end{eqnarray*}
           with $\tau_{w}= \sum _{h=1}^{w} \chi_{0}^{-h}$.
           For each $v \ge M+M^\nu$ the main value theorem gives us a real $c_v$ $v<c_v<v+1$ such that
       \begin{eqnarray*}
       &&     \left ( (v-M)^{2\alpha_{1}-1} v ^{2\alpha_{2}-1} - (v+1-M)^{2\alpha_{1}-1} 
           (v+1) ^{2\alpha_{2}-1}\right) =\\
           &=& (-c_v-M) ^{2\alpha_{1}-2} c_v ^{2\alpha_{2}-2} \left( (c_v-M) (2\alpha_{2}+1)
           +c_v (2\alpha_{1} +1) \right) 
             \end{eqnarray*}
         from this equality we infer 
         $$\left ( (v-M)^{2\alpha_{1}-1} v ^{2\alpha_{2}-1} - (v+1-M)^{2\alpha_{1}-1} 
           (v+1) ^{2\alpha_{2}-1}\right) =O\left( (v-M)^{{2\alpha_{1}-2}} v^{2\alpha_{2}-2}\right)$$
         and  
             \begin{eqnarray*}
        & & \Bigl \vert\sum_{v\ge M+M^\nu}
           \left( \left ( (v-M)^{2\alpha_{1}-1} v ^{2\alpha_{2}-1} - (v+1-M)^{2\alpha_{1}-1} 
           (v+1) ^{2\alpha_{2}-1}\right) \tau_{v}\right)\Bigr \vert =\\
           &=& \sum_{v\ge M+M^\nu} 
           O\left((v-M)^{2\alpha_{1}-2} v^{2\alpha_{2}-2}\right)= O\left( (M+S_{0}(M))^{2\alpha_{2}-1}\right) 
           =o(M^{2\alpha_{1}-1})
             \end{eqnarray*}
        Since
           $$\Bigl \vert  \left(M^\nu\right)^{2\alpha_{1}-1} (M+M^\nu)^{2\alpha_{2}-1}
           \tau_{S_{0}(M)-1} \Bigr \vert = o(M^{2\alpha_{1}-1})$$
           we have
           $\Sigma_{1}(M)=o(M^{2\alpha_{1}-1}),$ and 
           $\Sigma_1 (M) =O(\epsilon M^{2\alpha_{1}-1})$ for a sufficiently large $M$.
        The bounds $M>M^\nu>S_0$ implies 
           $$ \Sigma_{2}(M)= M^{2\alpha_{2}-1} C_{\alpha_{2}} \vert 1-\chi_{0}\vert ^{2\alpha_{1}} \chi_{0}^{-M}
          \left (1+R_2(M) \right)$$
          with $ R_2(M)=O(\epsilon))$. Then
           $-\alpha_2 <-\alpha_1$ 
           provides  
            $ \Sigma_{2}(M)=o(M^{2\alpha_1-1})$.
            The hypothesis on $M$ gives us                      
                      $$\Sigma_{3}(M) =C_{\alpha_{1}}C_{\alpha_{2}} \sum_{v=M^\nu}^{M-M^\nu} 
           (M-v)^{2\alpha_{1}-1} v^{2\alpha_{2}-1} 
           \chi_0^{-v} \left(1+R_3(M)\right) $$
           with $R_3(M) =O(\epsilon)$.
            Always with an Abel summation we  obtain  
          $$ \sum_{v=M^\nu}^{M-M^\nu} 
            (M-v)^{2\alpha_{1}-1} v^{2\alpha_{2}-1} 
           \chi_0^{-v} =  A_{1}+A_{2}$$
           with 
        $$
         A_{1}=C_{\alpha_{1}}C_{\alpha_{2}} \sum_{v=S_{0}(M)}^{M-S_{0}(M)} 
           \left( (M-v)^{-2\alpha_{1}-1} v^{2\alpha_{2}-1} - (M-v-1)^{2\alpha_{1}-1} (v+1)^{2\alpha_{2}-1} \right)
           \tau_{v}
           $$
           and 
           $$
          A_{2}= \tau_{S_{0}(M)-1} (M-S_{0}(M))^{2\alpha_{1}-1} 
          \left(S_{0}(M)\right)^{2\alpha_{2}-1} 
           - \tau_{M-S_{0}(M)-1} (M-S_{0}(M)) ^{2\alpha_{2}-1} 
           \left(S_{0}(M)\right)^{2\alpha_{1}-1}.
         $$
          As previously  for each integer $v$ such that 
            $M^\nu\le v \le M-M^\nu$ we have
             a real $c_v$ $v<c_v<v+1$ such  
                that 
                 \begin{eqnarray*}
                && \left( (M-v)^{2\alpha_{1}-1} v^{2\alpha_{2}-1} - (M-v-1)^{2\alpha_{1}-1} (v+1)^{2\alpha_{2}-1} 
                \right)=\\
                &=&O( c_v^{2\alpha_{2}-2} (M-c_v)^{2\alpha_{1}-2} )\le 
               O\left(\left( v (M-v)\right)^{2\alpha_{1}-2} \right).
                  \end{eqnarray*}
                  The study of the function 
                  $x\rightarrow x(M-x)$ on 
                  $[M^\nu, M -M^\nu]$ gives 
               $$   \left( (M-v)^{2\alpha_{1}-1} v^{2\alpha_{2}-1} - (M-v-1)^{2\alpha_{1}-1} (v+1)^{2\alpha_{2}-1}  \right)\le O(M^{2\alpha_1-2})
               =o(M^{2\alpha_1-1}).$$
                          Moreover it is easily seen that
      $$ A_{2}
           = o(M^{-2\alpha_1-1}).$$
           Hence for sufficiently large $M$ we may write 
           $\Sigma_3 (M) =O(M^{2\alpha_{1}-1})$. 
          We obtain also
           $$\Sigma_4(M) = C_{\alpha_1} M^{2\alpha_1-1} 
          \vert 1-\chi_0\vert^{2\alpha_{2}} \left(1+o(1)\right)$$
          and as for $\Sigma_1$ we have  
          $\Sigma_5 (M)=o(M^{2\alpha_1-1}).$
          Finally we have obtained the asymptotic 
          expansion
                                        $$\forall M \quad 
                    \mathrm{such} \quad \mathrm{that} \quad 
                     M^{\nu} >S_0 \quad \hat \sigma (M) =  C_{\alpha_1} M^{2\alpha_1-1} 
          \vert 1-\chi_0\vert^{-2\alpha_{2}} \left(1+O(\epsilon) \right),$$
          that was the aim of our demonstration.
\subsubsection{$n$ and $n+1$ singularities.}
Let 
$ \sigma =\vert 1 -\chi \vert ^{-2\alpha_{1}} \vert \prod_{j=2} ^n \vert  \chi - \chi_{0}\vert ^{-2\alpha_j} c$ with 
$c\in A(r,\mathbb T)$ $0<r<1$ and $-\alpha_1 > -\alpha_j,
\quad \forall j, \quad 2\le j\le n.$ Assume that 
for 
$\epsilon>0$ and a sufficiently large 
$M$ we have
      $$\hat \sigma (M) = C_{\alpha_1} \vert M \vert ^{2\alpha_1-1} c(1)
          \prod_{j=2} ^n \vert \chi_0 - 1\vert ^{-2\alpha_j} (1+R_M)
          $$ with $\vert R_M \vert \le \epsilon.$ 
        If 
         $ \sigma' =\vert 1 -\chi \vert ^{-2\alpha_{1}}  \prod_{j=2} ^{n+1} \vert \chi - \chi_{0}\vert ^{-2\alpha_j} c$  
           $c\in A(r,\mathbb T)$ $0<r<1$ and $\alpha_1 > \alpha_j,
\quad \forall j, \quad 2\le j\le n+1,$
           we  prove exactly as for the precedent point that $\sigma'$
           has the same property that $\sigma$, that ends the proof of the present lemma.
           \subsection{Proof of Theorem \ref{TROISIEME}
            and Corollary \ref{CDEUX}  }
 The proof is the same than for the theorem \ref{PREMIER}.
   We can write $T_N (\tilde f_1) = \tilde T_{1,N} +\tilde D_{1,N}$
   and $T_N (\tilde f_2) = \tilde T_{2,N} +\tilde D_{2,N}$,
   with if $k\not=l$ 
   $$  \left(  \tilde T_{1,N} \right)_{k+1,l+1} = 
c_1(1) C_{\alpha}
   \vert l-k \vert ^{2\alpha-1} \prod _{j=1}^p \vert 1-\chi_j\vert 
   ^{-2\alpha_j} $$
   $$  \left(  \tilde T_{2,N} \right)_{k+1,l+1} = 
c_2(1) C_{\beta}
   \vert l-k \vert ^{2\beta-1} \prod _{j=1}^q \vert 1-\chi_j\vert 
   ^{-2\beta_j} $$
  and\\
   $  \left(\tilde T_{1,N} \right)_{k+1,k+1} =0,$
   $  \left(\tilde T_{2,N} \right)_{k+1,k+1} =0.$
   Then $\tilde D_{1,N}$ and $\tilde D_{2,N}$ are Toeplitz matrices 
   with $(\tilde D_{1,N})_{k+1,l+1}= o \vert k-l \vert ^{2\alpha-1}$
   and $(\tilde D_{2,N})_{k+1,l+1}= o \vert k-l \vert ^{2\beta-1}$.
   hence we have (see \cite {BoVi}) 
   $\Vert \tilde D_{1,N} \Vert = o(N^{2\alpha})$ and 
     $\Vert \tilde D_{2,N} \Vert = o(N^{2\beta}).$
      As for the demonstration of Theorem \ref{PREMIER}
      we have 
      \begin{eqnarray*} 
      \Vert T_N (\tilde f_1)T_N(\tilde f_2) \Vert &=& 
      \Vert \tilde T_{1,N} \tilde T_{2,N} \Vert 
      + o(N^{2\alpha 2 \beta})   \\
      &=& C N^{2\alpha 2 \beta} \Vert K_{\alpha,\beta} \Vert 
      + o(N^{2\alpha +2 \beta})
      \end{eqnarray*}
      with
      $$ C = c_1(1) c_2 (1) C_\alpha C_\beta 
      \prod _{j=1}^p \vert 1-\chi_j\vert ^{-2\alpha_j} 
      \prod _{j=1}^q \vert 1-\chi_j\vert ^{-2\beta_j}.
      $$
    Corollary \ref{CDEUX} is a direct consequence 
      of the equality 
       $$ T_{N}(\vert \chi_{0}-\chi \vert ^{-2\alpha} \psi_1) = \Delta_{0}(\chi_{0}) T_{N}\left(  \vert 1-\chi \vert ^{-2\alpha} \psi_{1,\chi_{0}}\right)\Delta_{0}^{-1}(\chi_{0})$$
      and 
      $$ T_{N}(\vert \chi_{0}-\chi \vert ^{-2\beta} \psi_2) = \Delta_{0}(\chi_{0}) T_{N}\left(  \vert 1-\chi \vert ^{-2\beta} \psi_{2,\chi_{0}}\right)\Delta_{0}^{-1}(\chi_{0})$$
  where $\Delta_{0}(\chi_{0})$ is as in the introduction and
  $$ \psi_1 =\prod_{j=1}^p \vert \chi_{j}-\chi \vert ^{-2\alpha_{j}} c_{1}$$
$$ \psi_2 =\prod_{j=1}^q \vert \chi_{j}-\chi \vert ^{-2\beta_{j}} c_{2}$$
and $$ \psi_{1,\chi_{0}} (\chi) = \psi_1 (\chi_{0}\chi)
\quad \mathrm{and}\quad 
\psi_{2,\chi_{0}} (\chi) = \psi_2 (\chi_{0}\chi).$$

  \bibliography{Toeplitzdeux}
\end{document}